\documentclass[reqno]{amsart}
\usepackage{amssymb}  
\usepackage[T2A]{fontenc}
\usepackage[utf8]{inputenc}
\usepackage[russian,english]{babel}  
\usepackage{enumitem}\usepackage{cite}
\usepackage[dvipsnames]{xcolor}   
\usepackage{hyperref} 
\hypersetup{
    colorlinks=true,
    linkcolor=OliveGreen,  
    urlcolor=BlueViolet,
    citecolor=MidnightBlue
    }
\usepackage{tabularray} 

\newtheorem{thm}{Theorem}
\newtheorem{lem}{Lemma}
\newtheorem{prop}{Proposition}
\newtheorem{cor}{Corollary}

\begin{document}

\title[Refined conjugate generation]{Refined conjugate generation \\ in sporadic groups}
\author{Danila O. Revin}%
\address{Danila O. Revin
\newline\indent Sobolev Institute of Mathematics,
\newline\indent 4, Koptyug av.
\newline\indent 630090, Novosibirsk, Russia
\newline\indent {\sc and}
\newline\indent Krasovskii Institute of Mathematics and Mechanics,
\newline\indent 16, S. Kovalevskaja str.,
\newline\indent  620108, Yekaterinburg,  Russia
} 
\email{revin@math.nsc.ru
\newline\indent ORCID: \href{https://orcid.org/0000-0002-8601-0706}{0000-0002-8601-0706}}

\author{Andrei V. Zavarnitsine}%
\address{Andrei V. Zavarnitsine
\newline\indent Sobolev Institute of Mathematics,
\newline\indent 4, Koptyug av.
\newline\indent 630090, Novosibirsk, Russia
} 
\email{zav@math.nsc.ru
\newline\indent ORCID: \href{https://orcid.org/0000-0003-1983-3304}{0000-0003-1983-3304}
}
\thanks{This research was carried out within the State Contract of the Sobolev Institute of Mathematics (FWNF-2026-0017).}
\maketitle
\begin{quote}
\noindent{\sc Abstract. } 
Given an automorphism $x$ of order bigger than $2$ of a sporadic simple group $S$, we show that there are at most $3$ conjugates of $x$ required to generate a subgroup of order divisible by a fixed prime divisor $r$ of $|S|$. The only exception is the case where $S=Suz$, $x$ is in class $3A$, $r=11$, and then the required number of generators is~$4$.
\medskip

\noindent{\sc Keywords:} sporadic group, conjugacy, generators.

\medskip
\noindent{\small {\sc MSC2020:} 
20D06, 
20E28, 
20E45  
\hfill  {\sc UDC:}  512.542 }
\end{quote}

\section{Introduction}

Following~\cite{GS}, given a nonabelian finite simple group $S$ and its nonidentity automorphism $x$, we define $\alpha(x)=\alpha^{\phantom{S}}_{S}(x)$ to be the minimum number of conjugates of $x$ in $\langle S, x\rangle$ that generate a subgroup containing~$S$. This parameter has been broadly studied and found numerous applications, e.\,g., in \cite{DMPZ,fmow, RZ1, FGG, YRV, YWRV}. 

In this paper, we are interested in the following refinement introduced in \cite{YRV}.
Let $r$ be a prime divisor of $|S|$. We set $\beta_r(x) = \beta_{r,S}(x)$ to be the minimum number of conjugates of $x$ in $\langle S,x\rangle$ that generate a subgroup of order divisible by~$r$. The parameter $\beta_r(x)$, like $\alpha(x)$, can be used in obtaining analogues of the Baer-Suzuki theorem, see \cite{YRV,YWR,RZ1,WGR, YWRV}.  In some cases (e.\,g., in \cite[proof of Theorem~1]{RZ1} or \cite[proof of Proposition~2]{WGR}), knowing $\beta_r(x)$ for an appropriate $r$ makes it possible  to significantly improve the estimates on $\alpha(x)$, or even determine $\alpha(x)$ precisely.

Trivially, we have $\beta_r(x) \leqslant  \alpha(x)$ for all $r$ and $x$. Also,  $\beta_r(x)=1$ if and only if $r$ divides $|x|$, so we will assume throughout that $r$ is coprime to $|x|$. In particular, $\beta_r(x)=2$ whenever $\alpha(x)=2$.  It is known  \cite[Theorem~1]{R11} that $\beta_2(x)\leqslant 2$ for every simple group $S$ and $1\ne x\in\operatorname{Aut}S$, and so $\beta_2(x)= 2$ for $x$ of odd order. 
 
In the present paper, we restrict our attention to the case where $S$~is a sporadic group. 
It is known \cite[Theorem~1]{YWR} that in this case $\beta_3(x)\leqslant 3$ and $\beta_r(x)\leqslant r-1$ for $r>3$. Our main result, which significantly refines these estimates in the case $|x|>2$, is as follows.

\begin{thm}\label{main} Let $S$ be a sporadic group, let $x\in \operatorname{Aut} S$ with $|x| > 2$, and let $r$ be a prime divisor of $|S|$ coprime to $|x|$. Then $2\leqslant \beta_{r,S}(x)\leqslant 3$, except when $(S,x^S,r)=(Suz,3A,11)$ and then $\beta_{r,S}(x)=4$.
\end{thm}

We also find the precise value of $\beta_{r,S}(x)$ for many triples $(S,x^S,r)$ from Theorem~\ref{main} and list them in Table~\ref{tab:primes_bigger} below. 
The triples with yet unknown value of $\beta_r$ collected in Table~\ref{tab:beta_unknown} as well as the case $|x|=2$
are subject to further research.

The following fact is also a consequence of our results.

\begin{cor}\label{r235} Let $S$, $x$, and $r$ be as stated in Theorem~\rm{\ref{main}}. If $r=2,3$ or $5$ then $\beta_{r,S}(x)=2$.
\end{cor}

To prove these results we used the computer algebra system \texttt{GAP}~\cite{GAP}.
The accompanying code is available for verification from \cite{RZG}.

\begin{longtblr}[
     caption={Known values of $\beta_{r,S}(x)$ for $x\in \operatorname{Aut} S$ with $S$ sporadic simple, $|x|>2$,  and $r$ a prime divisor of $|S|$.},
     label={tab:primes_bigger},
      note{$\dagger$} = {Outer class.}
  ]{colspec = {l|l|l|l|l}, colsep = 3pt, rowsep = 0pt, 
    row{2}={c,abovesep=2pt,belowsep=1pt},
    row{3}={abovesep=3pt},
    column{2}={leftsep=4pt,rightsep=6pt},
    column{3-4}={leftsep=4pt},
    rowhead = 2}
  \hline[1pt]
  \SetCell[r=2]{m}  $S$      & \SetCell[r=2]{m,c}  $x^S$  & \SetCell[c=3]{c} $r$ \\ 
  \cline{3-5}
             &       & $\beta_{r}(x)=2$  & $\beta_{r}(x)=3$ & $\beta_r(x)=4$\\ 
\hline
 $J_2$        & $3A$             &   2, 5         & 7     & \\
 $McL$        & $3A$             &   2, 5         & 7, 11 & \\
 $Ly$         & $3A$             &   2, 5         &   & \\
 $Co_1$       & $3A$             &   2, 5         &   & \\
 $Fi_{22}$    & $3A$             &   2, 5         &   & \\
              & $3B$             &   2, 5, 7, 13  &  11 &  \\
 $Fi_{23}$    & $3A$             &   2, 5         &   & \\
              & $3B$             &   2, 5, 7, 13  &   & \\
 ${Fi_{24}}'$ & $3A$             &   2, 5         &   & \\
              & $3B$             &   2, 5, 7, 13  &   & \\
 $Suz$        & $3A$             &   2, 5         & 7, 13 & 11 \\
 $HS$         & $4A$             &   3, 5, 7      &  11 & \\
 $HN$         & $4D^\dagger$     &   3, 5, 7      &     & \\
\hline[1pt]
\end{longtblr}

\begin{table}[htb]
\caption{The triples $(S,x,r)$ with unknown $\beta_{r,S}(x)\in \{2,3\}$ for $S$, $x$, $r$ as in Table~\ref{tab:primes_bigger}.
\label{tab:beta_unknown}}
\begin{tabular}{l|l|l}
  \hline
   $S$        & $x^S$ &  $r\vphantom{S^{S^{S^S}}}$ \\ 
\hline
 $Ly$         & $3A$ &  7, 11, 31, 37, 67  $\vphantom{S^{S^S}}$ \\
 $Co_1$       & $3A$ &  7, 11, 13, 23  \\
 $Fi_{22}$    & $3A$ &  7, 11, 13  \\
 $Fi_{23}$    & $3A$ &  7, 11, 13, 17, 23  \\
              & $3B$ &  11, 17, 23  \\
 ${Fi_{24}}'$ & $3A$ &  7, 11, 13, 17, 23, 29  \\
              & $3B$ &  11, 17, 23, 29  \\
 $HN$         & $4D$ &  11, 19  \\
\hline
\end{tabular}
\end{table}

\section{Preliminaries}

As in \cite{DMPZ}, slightly abusing the notation we will often
write $\alpha(nX)$ instead of $\alpha(x)$ if $x\in nX$, and similarly for $\beta_r$, where $nX$ denotes a conjugacy class of $G$ consisting of elements of order $n$ and $X$ is an appropriate letter chosen according to the labelling in \cite{atlas} or \cite{GAP}.

Given conjugacy classes $C_1$, \ldots, $C_n$ of a group $G$ and an element $x_n\in C_n$, denote by $\operatorname{m} (C_1,\ldots,C_n)$ the number of $(n-1)$-tuples
$(x_1,\ldots,x_{n-1})$ of elements of $G$ with $x_1\ldots x_{n-1}=x_n$, where $x_i\in C_i$ for all $i$.
From character theory, it is known that this number can be found via the ordinary character table of $G$ using the formula\footnote{This formula is used in \cite[p. 885]{DMPZ} with neither a proof nor a reference. A proof 
can be found in \cite[Section 5.3]{mm} where (\ref{mcc}) is stated in a slightly different form.}
\begin{equation}\label{mcc}
\mathrm{m}(C_1,\ldots,C_n)=\frac{|C_1|\ldots|C_{n-1}|}{|G|}\sum\limits_{\chi\in\mathrm{Irr}(G)}\frac{\chi(x_1)\ldots\chi(x_{n-1})\overline{\chi(x_n)}}{\chi(1)^{n-2}}.
\end{equation}

A similar but slightly different parameter $\operatorname{n}(C_1,\ldots,C_n)$ is the number of $n$-tuples $(x_1,\ldots,x_n)$ with $x_1\ldots x_n=1$ and $x_i\in C_i$ for all $i$. It is given by 
$$
\mathrm{n}(C_1,\ldots,C_n)=\frac{|C_1|\ldots
|C_n|}{|G|}\sum\limits_{\chi\in\mathrm{Irr}(G)}\frac{\chi(x_1)\ldots\chi(x_n)}{\chi(1)^{n-2}}.
$$
These values are related by the formula
$$
|C_n|\cdot\mathrm{m}(C_1,\ldots,C_n)= \mathrm{n}(C_1,\ldots,\overline{C_n}),
$$
where $\overline{C_n}$ denotes the class containing the inverses of the elements in $C_n$. For our purposes, both $\mathrm{m}(C_1,\ldots,C_n)$ and $\mathrm{n}(C_1,\ldots,C_n)$ can be determined either from the Atlas \cite{atlas} or using \texttt{GAP} which has relevant built-in functions

\smallskip
\texttt{ClassMultiplicationCoefficient()} 

\smallskip\noindent 
to find $\mathrm{m}(\,)$ for $n=3$  and 

\smallskip
\texttt{ClassStructureCharTable()}

\smallskip\noindent 
to find $\mathrm{n}(\,)$ for arbitrary $n$.

The following result is a consequence of \cite[Lemma 2.3]{DMPZ} and a comment thereafter.

\begin{prop}\label{DiM_Zal} In a centreless group $G$, let $C_1,\ldots, C_n$ be conjugacy classes and let $x_i\in C_i$ be such that $x_1\ldots x_{n-1}=x_n$.
If $\mathrm{m}(C_1,\ldots,C_n)<|C_G(x_n)|$ then $\langle x_1,\ldots,x_{n-1}\rangle$ is a proper subgroup of $G$.
\end{prop}

Let a group $G$ act on a vector space $V$.
Following L.\,L. Scott \cite{Scott}, for a subgroup or an element $X$ of $G$, we denote by $\nu(X)=\nu(X,V)$ the codimension of the fixed-point subspace of $X$ on $V$, and also write $\nu(X^*)$ for $\nu(X,V^*)$, where $V^*$ is the dual of $V$.
\begin{prop}[{\cite[Theorem 1]{Scott}}] \label{scott_thm}
Suppose $G$ is generated by $x_1,\ldots,x_n$ with $x_1\ldots x_n=1$. Then 
$$
\sum_{i=1}^n \nu(x_i) \geqslant \nu(G) + \nu(G^*).
$$
\end{prop}

\begin{cor}\label{scott_cor}
Let $F$ be a field and let $V$ be an irreducible nonprincipal $FG$-module.
Suppose $G$ is generated by $x_1,\ldots,x_n$ with $x_1\ldots x_n=1$.
Then 
\begin{equation}\label{scott_ineq}
\sum_{i=1}^n\operatorname{dim}C_V(x_i)
\leqslant (n-2)\operatorname{dim}V.
\end{equation}
\end{cor}
\begin{proof} We have $\nu(G)=\nu(G^*)=\operatorname{dim}V$ because $V$ is nonprincipal and irreducible.
Also, $\nu(x_i)=\operatorname{dim}V-\operatorname{dim}C_V(x_i)$ for all $i$ by definition. The claim now follows from Proposition~\ref{scott_thm}.
\end{proof}

The next result about certain values of $\alpha$ for sporadic almost simple groups is a consequence of \cite[Theorem 3.1(1)]{DMPZ}, \cite[Table 2]{fmow}, and, in part, \cite{RZ25} and the references therein.

\begin{prop}\label{alpha} Let $S$~be a sporadic group and let $x\in \operatorname{Aut}S$ with $|x|>2$. Then $\alpha^{\phantom{S}}_{S}(x)=2$, except in the following cases:
    \begin{itemize}
         \item[$\bullet$] $(S,x^S)\in\big\{(J_2,3A)$, $(HS,4A)$, $(McL,3A)$, $(Ly,3A)$, $(Co_1,3A)$, $(F_{22},3A)$, \\ $(Fi_{22},3B)$,
         $(Fi_{23},3A)$, $(Fi_{23},3B)$, $(Fi_{24}',3A)$, $(Fi_{24}',3B)$, $(HN,4D)\big\}$ and $\alpha^{\phantom{S}}_{S}(x)=3$;
     \item[$\bullet$] $(S,x^S)=(Suz,3A)$ and $\alpha^{\phantom{S}}_{S}(x)=4$,
    \end{itemize}
\end{prop}

For our purpose, we will also need the following values of $\alpha$ for specific groups.

\begin{lem}\label{more_alpha}
If $(S,x^S)\in\big\{(U_5(2),3C),(U_5(2),3D),(U_6(2),3B)\big\}$ then 
$\alpha^{\vphantom{A}}_S(x)=3$.
\end{lem}
\begin{proof}  We treat the groups $U_5(2)$ and $U_6(2)$ separately.

$\bullet$ {\em Case $S=U_5(2)$.} Classes $3C$ and $3D$ of $S$ are mutually inverse, so we may consider either. Calculation \cite[Sec.~2]{RZG} shows  that  $\operatorname{m}(3C, 3C, 12A )>0$
and $\operatorname{m}(3C, 12A, 11A )>0$. Hence, there are three elements in $3C$ that generate a subgroup $H$ with elements of both orders $11$ and $12$. By \cite{atlas}, the only maximal subgroup of $S$ of order divisible by $11$ is $L_{2}(11)$ which has no elements of order $12$. Therefore, $H=S$ and $\alpha(3C)\leqslant 3$.

Next, running through all\footnote{We need not exclude the classes $nX$ with $\operatorname{m}(3C,3C,nX)=0$ because the required inequality is trivially satisfied in these cases.} conjugacy classes $nX$ of $S$, we check in \cite[Sec.~3]{RZG} that the inequality $\operatorname{m}(3C,3C,nX)<|C_S(x)|$ always holds, where $x\in nX$. Proposition~\ref{DiM_Zal} now implies that $S$ cannot be generated by a pair of elements in $3C$, and so ${\alpha(3C)=3}$.  

$\bullet$ {\em Case $S=U_6(2)$.} First, we show that  $\alpha^{\vphantom{A}}_S(x)\leqslant 3$ for $x\in 3B$. In terms of standard generators $a$ and $b$ of orders $2$ and $7$, $S$ is represented in \texttt{GAP} as a permutation group of degree $672$. We check in \cite[Sec.~4]{RZG} that $x=(ab^2)^6$ lies in class $3B$ and $x$, $x^a$, and $x^b$ generate the whole of $S$. Hence,  $\alpha^{\vphantom{A}}_S(x)\leqslant 3$.

Suppose that $\alpha^{\vphantom{A}}_S(x)=2$ and let $S=\langle x,y\rangle$ with $x,y\in 3B$. Calculation \cite[Sec.~5]{RZG} shows that  $\operatorname{m}(3B,3B,nX)\geqslant |C_S(t)|$ with $t\in nX$ only when $nX$ is $7A$ or $9C$. We consider these two possibilities separately by lifting the problem to $T=SU_6(2)$ and then applying (\ref{scott_ineq}).

$\diamond\ $ {\em Subcase $nX=7A$.}  We have $xy\in 7A$. Class $3B$ of $S$ lifts to classes $3f$, $3g$, and $3h$ of $T$, and class $7A$ lifts to $7a$, $21a$, and $21b$,  see \cite{atlas}. Let $x_0,y_0\in 3f$ be preimages in $T$ of $x$ and $y$. Then $\langle x_0,y_0\rangle=T$ and, necessarily, $z_0=x_0 y_0\in 7a$ because in $T$ we have
$\operatorname{m}(3f,3f,21a)=\operatorname{m}(3f,3f,21b)=0$,
see \cite[Sec.~6]{RZG}. Let $V$ be either of the two $15$-dimensional irreducible $\overline{\mathbb{F}}_2T$-modules with Brauer character~$\varphi$ \cite{GAP} whose values on some conjugacy classes of $T$ are given in (\ref{br15}).
\begin{equation}\label{br15}
    \begin{tabular}{c|cccc}
     $g$    & $1a$ & $3f$ & $7a$ & $9g$ \\
         \hline
    $\varphi(g)^{\vphantom{A^A}}$  & 15  & 6 & 1 & 0
    \end{tabular}
\end{equation}
Then the sum 
\begin{equation}\label{cx0y0z0}
\operatorname{dim}C_V(x_0)+\operatorname{dim}C_V(y_0)+\operatorname{dim}C_V(z_0)=2(\varphi_{\langle x_0\rangle},1_{\langle x_0\rangle})+(\varphi_{\langle z_0\rangle},1_{\langle z_0\rangle})
\end{equation}
equals
$$
\frac{2}{3}\big(\varphi(1a)+2\cdot\varphi(3f)\big)+\frac{1}{7}\big(\varphi(1a)+6\cdot\varphi(7a)\big)\\
=18+3=21.
$$
This is bigger than $\operatorname{dim}V=15$ which contradicts the fact that $T=\langle x_0,y_0\rangle$ by Corollary~\ref{scott_cor}.

$\diamond\ $ {\em Subcase $nX=9C$.} In this subcase, $xy$ is in class $9C$ which lifts to a unique class $9g$ of $T$. Therefore, $z_0=x_0 y_0\in 9f$, where  
as above $x_0,y_0\in 3f$ are preimages in $T$ of $x,y$. Considering 
the same $15$-dimensional $\overline{\mathbb{F}}_2T$-module $V$ as above we obtain that the sum in (\ref{cx0y0z0}) equals
$$
\frac{2}{3}\big(\varphi(1a)+2\cdot\varphi(3f)\big)+\frac{1}{9}\big(\varphi(1a)+2\cdot\varphi(3f)+0\cdot\varphi(9g)\big)=18+3=21,
$$
because $9g$ cubes to $3f$, and this also contradicts Corollary~\ref{scott_cor}. Therefore, $\alpha^{\vphantom{A}}_S(x)=3$ as claimed.
\end{proof}

\section{Proof of main result}

We now prove Theorem~\ref{main}.

\begin{proof} Let $x\in \operatorname{Aut} S$ with $|x|>2$.
We may assume that $\alpha(x)>2$. By Proposition~\ref{alpha}, this means that $S$ and $x^S$ are
as listed in the first two columns of Table~\ref{tab:primes_bigger}.

If it happens that $\operatorname{m}(x^S,x^S,y^S)>0$ for some $y\in S$ of order divisible by $r$, where $r$ is a prime divisor of $|S|$ such that $r\nmid |x|$, then $\beta_{r,S}(x)=2$. All such cases can be found uniformly using \texttt{GAP}, see \cite[Sec.~1]{RZG}. They are listed in column 3 of Table~\ref{tab:primes_bigger}.

We now consider the cases from columns 4 and 5 of Table~\ref{tab:primes_bigger} one by one. A working method here is determining the class fusion of subgroups in a group. In all cases below, this can be done using \texttt{GAP}'s function \texttt{PossibleClassFusions}, see \cite{RZG} for detailed calculations. We also use explicit constructions for sporadic almost simple groups in terms of the so-called {\em standard} generators \cite{Wil96} as mentioned in the online atlas \cite{AtlRep}.

$\bullet$ $(S,x^S,r)= (Suz,3A,7)$. There are three elements in class $3A$ of $Suz$
whose product has order $7$. This is because $\operatorname{n}(3A,3A,3A,7A)>0$, see \cite[Sec.~10]{RZG}. Therefore, $\beta_7(3A)\leqslant 3$. Suppose that $\beta_7(3A)=2$.
Let $H$ be a subgroup of $S$ of order divisible by $7$ generated by two elements $x,y$ in $3A$. Since $\alpha(x) > 2$ (in fact, $\alpha(x)=4$) by Proposition~\ref{alpha}, $H$ is a proper subgroup contained in a maximal subgroup $M<S$ of order divisible by $7$. By \cite{atlas}, there are $5$ possibilities for $M$: 
$$G_2(4), \ \ 3_2 ^{\ {\displaystyle\cdot}\,} U_4(3)\!:\!2'_3, \ \ J_2\!:\!2, \ \ (A_4 \times L_3(4) )\!:\!2, \ \ A_7.$$

{\em Case $M=G_2(4)$.} Only class $3a$ of $G_2(4)$ fuses to $3A$ of $Suz$, see \cite[Sec.~11(a)]{RZG}. Using standard generators $a$ and $b$ of $G_2(4)$, we check \cite[Sec.~13]{RZG} that  $c=(abab^2)^5$ is a representative of class $3a$ and the order of any $(3a,3a)$-generated subgroup $\langle c, c^g\rangle$, $g\in M$, is not divisible by $7$ (or $13$, which we will use below). Hence, this case is impossible.

{\em Case $M=3_2 ^{\ {\displaystyle\cdot}\,} U_4(3)\,\mathord{:}\,2'_3$.} This group has $5$ conjugacy classes of elements of order~$3$. 
In Table~\ref{cls3}, we list their sizes, as well as the class sizes for the  elements of order $3$ in $3_2 ^{\ {\displaystyle\cdot}\,} U_4(3)$ and $U_4(3)$, see \cite[Sec.~16]{RZG}. In view of \cite[Sec.~11(b)]{RZG}, only classes $3a$ and $3b$ of $M$ fuse to $3A$ of $Suz$, hence $x,y$ belong to $3a\cup 3b$. Since $3a$ has size $2$, it lies in the $3$-radical $R$ of~$M$. This means that neither $x$ nor $y$ is in $3a$, or else the subgroup
$\langle x,y \rangle$ could not have order divisible by~$7$, because its image in the quotient by $R$ would be a group of order dividing~$3$.

\begin{table}[htb]
\caption{Class sizes of elements of order $3$ in sections of  
$3_2 ^{\ {\displaystyle\cdot}\,} U_4(3)\,\mathord{:}\,2'_3$.\label{cls3}}
\begin{tabular}{|l|r|c@{\quad}|l|r|c@{\quad}|l|r|}
\cline{1-2} \cline{4-5} \cline{7-8} 
\multicolumn{2}{|c|}{$3_2 ^{\ {\displaystyle\cdot}\,} U_4(3)\,\mathord{:}\,2'_3$}&&
\multicolumn{2}{|c|}{$3_2 ^{\ {\displaystyle\cdot}\,} U_4(3)$}&&
\multicolumn{2}{|c|}{$U_4(3)$ $\vphantom{S^{S^S}g_{g_{g_g}}}$} \\
\cline{1-2} \cline{4-5} \cline{7-8}
class&\multicolumn{1}{l|}{size}&&class&\multicolumn{1}{l|}{size}&&class& \multicolumn{1}{l|}{size$\vphantom{S^{S^S}}$}\\
\cline{1-2} \cline{4-5} \cline{7-8}
$3a$ & $2$      && $3a$, $3b$        & $1$      && $3a$       &  $  560\vphantom{S^{S^S}}$ \\
$3b$ & $560$    && $3c$, $3d$, $3e$  & $560$    && $3b$, $3c$ &  $ 3360$ \\
$3c$ & $1120$   && $3f$, $3g$        & $10080$  && $3d$       &  $40320$ \\
$3d$ & $20160$  && $3h$              & $120960$ &&            &          \\
$3e$ & $120960$ &&                   &          &&            &          \\
\cline{1-2} \cline{4-5} \cline{7-8}
\end{tabular}
\end{table}

Thus, $x$ and $y$ are both in class $3b$ of $M$. From the class sizes, we see that class $3b$ of $M$ is one of the classes $3c$, $3d$, or $3e$ of  
$3_2^{\ {\displaystyle\cdot}\,} U_4(3)$   (the other two being merged 
to $3c$ of $M$ by the outer automorphism $2'_3$). All classes  $3c$, $3d$, and $3e$ of $3_2 ^{\ {\displaystyle\cdot}\,}U_4(3)$ map to $3a$ of $U_4(3)$
under the natural epimorphism $3_2^{\ {\displaystyle\cdot}\,} U_4(3)\to U_4(3)$. In particular, the images $\overline{x},\overline{y}\in 3a$ generate in $U_4(3)$ a subgroup of order divisible by~$7$.

The group $U_4(3)$ has standard generators $a$ and $b$ of orders $2$ and $6$. We check \cite[Sec.~14]{RZG} that $c=b^2$ has centraliser of order $5832$, which implies that $c\in 3a$. 
By considering the subgroups $\langle c,c^g\rangle$, $g\in U_4(3)$, we conclude that no $(3a,3a)$-generated subgroup of $U_4(3)$ has order divisible by $7$, a contradiction.

{\em Case $M=J_2\!:\!2$.} In $M$, there are two conjugacy classes $3a$ and $3b$ of elements of order $3$, but only $3a$ fuses to $3A$ of $Suz$, see \cite[Sec,~11]{RZG}. Since $x,y$ lie in the subgroup $J_2$ of $M$, so does $H = \langle x,y \rangle$ and, inasmuch as $\alpha^{\vphantom{A}}_{J_2}(3a)=3$ (see \cite[Theorem 3.1(1a)]{DMPZ}), $H$ is a proper subgroup of $J_2$.  So, $H$ lies in a maximal subgroup $M_1$ of $J_2$ of order divisible by $7$. By \cite{atlas}, $M_1$ is either $U_3(3)$ or $\operatorname{PGL}_2(7)$.

Suppose that $M_1=U_3(3)$. Class fusion shows \cite[Sec.~12]{RZG} that $3a$ is the only class in $M_1$ that fuses to $3a$ of $J_2$. The standard generators $a$ and $b$ have orders $2$ and $6$, and the centraliser of $c=b^2$ has size $108$, which implies that $c\in 3a$.
Running through all subgroups generated by a pair of elements in $c^{M_1}$, we see \cite[Sec.~15]{RZG} that none of them has order divisible by $7$, a contradiction.

Therefore, $M_1=\operatorname{PGL}_2(7)$. We also check in \cite[Sec.~12]{RZG} that the elements of order $3$ in $\operatorname{PGL}_2(7)$ fuse to class $3b$ of $J_2$ rather than $3a$. This contradiction shows that $M$ cannot be $J_2\!:\!2$.

{\em Case $M=(A_4 \times L_3(4) )\!:\!2$.} In this case, $M$ has three classes $3a$, $3b$, and $3c$ of elements of order $3$. The sizes of these classes are $2240$, $8$, and $17920$, respectively, see \cite[Sec.~17]{RZG}. Observe that $2240$ is also the size of (a unique) class $C_2$ of elements of order $3$ in $L_3(4)$, while $8$ is the sum of sizes of two classes of elements of order $3$ in $A_4$ which are merged to $C_1$ by outer involution. It follows from the structure of $M$
that $3a$ consists of the elements $(1,y)\in A_4 \times L_3(4)$, $y$ in $C_2$, and 
$3b$ consists of the elements $(x,1)\in A_4 \times L_3(4)$, $x$ in $C_1$.

Class fusion shows \cite[Sec.~11(f)]{RZG} that only class $3b$ fuses to $3A$ of $Suz$. However, the above description implies that every pair of elements of $3b$ generates a subgroup of $A_4$ whose order is not divisible by $7$, a contradiction.

{\em Case $M=A_7$.} It follows from class fusion \cite[Sec.~11(j)]{RZG} that no elements of order $3$ of $A_7$ fuse to class $3A$ of $Suz$, a contradiction.

Consequently, we have  
$$\beta_{7,Suz}(3A)=3.$$

$\bullet$ $(S,x^S,r)=(Suz,3A,11)$. By Proposition~\ref{alpha}, we have $\beta_{11}(3A)\leqslant 4$. Assume that $\beta_{11}(3A)\leqslant 3$ and let 
$H$ be the subgroup of $Suz$ of order divisible by $11$ generated by $3$ (not necessarily distinct) elements $x,y,z\in 3A$. Since $\alpha(x)=4$, $H$ is a proper subgroup contained in a maximal subgroup $M < S$ of order divisible by $11$.
By \cite{atlas}, there are three possibilities for $M$: 
$$U_5(2), \ \ 3^5\!:\!M_{11}, \ \ M_{12}\!:\! 2.$$

{\em Case $M=U_5(2)$.} Class fusion shows \cite[Sec.~11(c)]{RZG} that 
only elements of classes $3a$ and $3b$ in $U_5(2)$ fuse to $3A$ in $Suz$. Consequently, $x,y,z$ lie in $3a\cup 3b$. Observe that $3a$ and $3b$ of $U_5(2)$ are merged by an outer automorphism 
into a single class $3a$ of $U_5(2).2$. So, is suffices to check whether any $(3a,3a,3a)$-generated subgroup of $U_5(2).2$ has order divisible by $11$. Using the standard generators $a$ and $b$ of $U_5(2).2$, where  $|a|=2$ and $|b|=4$, we show \cite[Sec.~18]{RZG} that $c=(ab(ab^2)^2)^4$ has order $3$ and centraliser $C$ in $U_5(2).2$ of size $77760$. In particular,  $c$ is in class $3a$ of $U_5(2).2$.  By considering representatives $(g,h)$ of  the orbits of $C$ on the pairs in $3a\times 3a$, we find that every group $\langle c,g,h\rangle$ is a $\{2,3\}$-group and, therefore, has order not divisible by $11$, a contradiction.

{\em Case $M=3^5\!:\!M_{11}$.} It follows from class fusion \cite[Sec.~11(d)]{RZG}
that only class $3a$ of $M$ fuses to $3A$ of $Suz$; hence, $x,y,z$ lie in $3a$ of $M$.
However, class $3a$ (and also $3b$) of $M$ lies in the $3$-radical of $M$. This can be seen from the character table \cite[Sec.~19]{RZG}, where the first $10$ irreducible  characters of $M$ 
are unfaithful with classes $3a$ and $3b$ in their kernels (they are the lifts to $M$ of the $10$ irreducible characters of $M_{11}$). In particular, $\langle x,y,z\rangle$ cannot have order divisible by $11$, a contradiction.

{\em Case $M = M_{12}\!:\! 2$.} From class fusion, we see \cite[Sec.~11(g)]{RZG} that no elements of order $3$ in $M$ fuse to $3A$ of $Suz$, a contradiction.

Consequently, we obtain 
$$\beta_{11,Suz}(3A)=\alpha^{\vphantom{A}}_{Suz}(3A)=4.$$

$\bullet$ $(S,x^S,r)=(Suz,3A,13)$.
Calculation \cite[Sec.~10]{RZG} shows that there are three elements in class $3A$ of $Suz$
whose product has order $13$. This holds because $\operatorname{n}(3A,3A,3A,13A)>0$. Therefore, $\beta_{13}(3A)\leqslant 3$. 

Suppose that $\beta_{13}(3A)=2$. Let $H$ be a subgroup of $S$ of order divisible by $13$ generated by $x,y\in 3A$. Since $\alpha(x) > 2$, $H$ is a proper subgroup contained in a maximal subgroup $M<S$ of order divisible by $13$. By \cite{atlas}, there are 
three possibilities for $M$: 
$$G_2(4), \ \  L_3(3):2, \ \ L_2(25).$$

{\em Case $M = G_2(4)$}. As we have mentioned in the case $r=7$ above, only class $3a$ of $M$
fuses to $3A$ of $Suz$ and there are no $(3a,3a)$-generated subgroups of $M$ of order divisible by $13$, a contradiction.

{\em Cases $M = L_3(3)\!:\!2$ and $L_2(25)$}. Using class fusion we check \cite[Secs.~11(h), 11(i)]{RZG} that no elements of order $3$ of $M$ fuse to $3A$ of $Suz$, a contradiction.

Therefore, we have 
$$\beta_{13,Suz}(3A)=3.$$

$\bullet$ $(S,x^S,r)=(Fi_{22},3B,11)$. From Proposition~\ref{alpha}, it follows that $\beta_{11}(x)\in \{2,3\}$. Assume that $\beta_{11}(x)=2$. Then $H=\langle x, y\rangle$ has order divisible by $11$ for some $y\in 3B$. Such an $H$ must be a proper subgroup because $\alpha(x)=3$, and so $H\leqslant M<S$, where $M$ is a maximal subgroup of order divisible by $11$.
According to \cite{atlas}, $M$ is isomorphic to one of 
$$ 2^{\textstyle \cdot}U_6(2), \ \ 2^{10}\!:\!M_{22},  \ \text{or} \ M_{12}.$$

Suppose $M\cong 2^{10}\!:\!M_{22}$. Note that $M$ has a unique conjugacy class $3a$ of elements of order $3$, because so does $M_{22}$. We check \cite[Sec.~7(b)]{RZG} that this class fuses to class $3C$ of $S$ rather than $3B$. Hence $x\not\in M$, a contradiction.

Suppose $M\cong M_{12}$. Then classes $3a$ and $3b$ of $M$ fuse to classes $3D$ and $3C$ of~$S$, respectively, \cite[Sec.~7(c)]{RZG}. Again, $x$ cannot be in $M$ because $x\in 3B$, a contradiction.

Therefore, $M\cong 2^{\textstyle \cdot}U_6(2)$. Only class $3b$ of $M$ fuses into $3B$, see \cite[Sec.~7(a)]{RZG}, and $3b$ of $M$ maps to class $3b$ of $M_0=U_6(2)$ under the natural epimorphism $\tau: M\to M_0$, see \cite{atlas}. Therefore, $H_0=\langle x_0,y_0\rangle \leqslant M_0$ has order divisible by $11$, where $x_0=\tau(x)$ and $y_0=\tau(y)$.

By Lemma~\ref{more_alpha}, we have $\alpha^{\vphantom{A}}_{M_0}(3b)=3$ which implies that $H_0$ is a proper subgroup of $M_0$, and so $H_0\leqslant M_1$ for a maximal subgroup $M_1<M_0$ whose order is divisible by $11$. Up to isomorphism, $M_1$ is either $U_5(2)$ or $M_{22}$, see \cite{atlas}. However, $3a$, the only class of $M_{22}$ of elements of order $3$, fuses to $3c$ of $M_0$ rather than $3b$, see \cite[Sec.~8(b)]{RZG}. Therefore, $M_1\cong U_5(2)$. Again, analysing class fusion \cite[Sec.~8(a)]{RZG} we see that classes $3c$ and $3d$ of $M_1$ fuse to $3b$ of $M_0$.

Since $\alpha^{\vphantom{A}}_{M_1}(3c)=\alpha^{\vphantom{A}}_{M_1}(3d)=3$ by Lemma~\ref{more_alpha}, we conclude that $H_0$ is a proper subgroup of $M_1$, and so $H_0\leqslant M_2$ for a maximal subgroup $M_2<M_1$ of order divisible by $11$. Up to isomorphism, the only possibility is $M_2\cong L_2(11)$. However the only class $3a$ of $M_2$ of elements of order $3$ fuses to $3f$ of $M_1$, rather than $3c$ or $3d$, see \cite[Sec.~9]{RZG}. This final contradiction shows that 
$$\beta_{11,Fi_{22}}(3B)=\alpha^{\vphantom{A}}_{Fi_{22}}(3B)=3.$$

$\bullet$ $(S,x^S,r)= (J_2,3A,7)$. The Janko group $J_2$ has standard generators $a$ and $b$, where $a\in 2B$ and $b\in 3B$. We check in \cite[Sec.~20]{RZG} that $c=(abab^2)^4$ is in class $3A$, and all subgroups $\langle c,d\rangle$, where $d$ runs through representatives of the orbits of $C_S(c)$ on $3A$, have order not divisible by $7$. Therefore, 
$$\beta_{7,J_2}(3A)=\alpha^{\vphantom{A}}_{J_2}(3A)=3.$$

$\bullet$ $(S,x^S,r)= (HS,4A,11)$. We show in \cite[Sec.~21]{RZG} that in terms of the standard generators $a$ and $b$ of $HS$, where $a\in 2A$ and $b\in 5A$, the element $c=(ab(ab^3)^2)^3$ lies in $4A$ and all subgroups $\langle c,c^g\rangle$, $g\in S$, have order not divisible by $11$. Therefore, 
$$\beta_{11,HS}(4A)=\alpha^{\vphantom{A}}_{HS}(4A)=3.$$

$\bullet$ $(S,x^S)= (McL,3A)$, $r=7,11$. The McLaughlin group $McL$ has standard generators $a$ and $b$, where $a\in 2A$ and $b \in 5A$. We check in \cite[Sec.~22]{RZG} that $c=(ab^2)^4$ has order $3$ and centraliser of order $29160$. Thus, $c\in 3A$. Again, it is shown in \cite[Sec.~22]{RZG} that all subgroups $\langle c,c^g\rangle$, $g\in S$, have order not divisible by~$7$ or $11$. It follows that 
$$\beta_{r,McL}(3A)=\alpha^{\vphantom{A}}_{McL}(3A)=3.$$
\end{proof}

In Table~\ref{tab:beta_unknown}, we have collected the prime divisors $r$ of $|S|$ that do not occur in Table~\ref{tab:primes_bigger}. The estimates $2\leqslant\beta_r(x)\leqslant 3$ for those primes $r$ are a consequence of Proposition~\ref{alpha}. Corollary~\ref{r235} now also follows from Table~\ref{tab:primes_bigger}.

\section{Final remarks and open problems}

It is conjectured in \cite[Conjecture~2]{YRV} that, for every nonabelian simple group $S$ of order divisible by an odd prime~$r$ and a nonidentity automorphism $x$ of $S$, the following inequality holds:
\begin{equation}\label{hyp}
\beta_{r,S}(x)\leqslant
\left\{
\begin{array}{ll}
  3,   & \text{if } r=3,  \\
  r-1,   & \text{if } r>3.
\end{array}
\right.
\end{equation}
At present, this conjecture is confirmed for many simple groups (see \cite{YRV,YWR,YWRV,RZ1,WGR}, etc.) including the sporadic groups \cite{YWR} as was mentioned above.  For all $r$, there are examples where the equality in (\ref{hyp}) is attained. For instance, $\beta_{3,S}(x)=3$ if $S=L_2(27)$ and $x\in S$ is an involution, and $\beta_{r,S}(x)=r-1$ if $S=A_r$ for $r>3$ and $x\in S_r$ is a transposition. In all known such examples, $|x|=2$. It is natural to conjecture that the equality is attainable only when  $x$~is an involution. The results of the present paper show that this is indeed so for the sporadic groups.

To study the sharp $\pi$-analogue of the Baer--Suzuki theorem \cite[Conjecture~1]{YRV} it would be interesting and useful to classify all pairs $(S,x^S)$ for which the equality in (\ref{hyp}) holds. The case $r=3$ is especially interesting, see \cite[Problem~21.111]{Kour}.


\begin{thebibliography}{10}

\bibitem{atlas} J.~H.~Conway, R.~T.~Curtis, S.~P.~Norton, R.~A.~Parker, R.~A.~Wilson, Atlas of finite groups: maximal subgroups and ordinary characters for simple groups. Oxford. Clarendon Press (1985), xxxiii + 252 pp.

\bibitem{DMPZ} L. Di\,Martino, M.\,A. Pellegrini, A.\,E. Zalesski, On generators and representations of the sporadic simple groups. \textit{Comm. Algebra}, {\bf 42}, N\,2 (2014), 880--908.

DOI: \href{https://doi.org/10.1080/00927872.2012.729629}{\texttt{10.1080/00927872.2012.729629}}

\bibitem{fmow} J. Fawcett, J. M\"uller, E. O'Brien, R. Wilson, Regular orbits of sporadic simple groups, {\em J. Algebra}, {\bf 522} (2019) 61--79.
DOI: \href{https://doi.org/10.1016/j.jalgebra.2018.11.034}{\texttt{10.1016/j.jalgebra.2018.11.034}}

\bibitem{FGG} P. Flavell,  S. Guest, R. Guralnick, Characterizations of the solvable radical, \textit{Proc. Amer. Math. Soc.}, {\bf 138}:4 (2010), 1161--1170. DOI: \href{https://doi.org/10.1090/S0002-9939-09-10066-7}{\texttt{10.1090/S0002-9939-09-10066-7}}

\bibitem{GAP} The GAP Group, GAP --- Groups, Algorithms, and Programming, Version 4.13.0 (2024). URL: \url{http://www.gap-system.org}

\bibitem{GS} R. Guralnick, J. Saxl, Generation of finite almost simple groups by conjugates,
\textit{J. Algebra}, {\bf 268} (2003), 519--571. DOI: \href{https://doi.org/10.1016/S0021-8693(03)00182-0}{\texttt{10.1016/S0021-8693(03)00182-0}}

\bibitem{Kour} E.\,I. Khukhro, V.\,D. Mazurov (ed.). {\em Unsolved Problems in Group Theory. The Kourovka Notebook.} No. 21. Novosibirsk, 2026,  electronic, 299 pp. URL: \url{https://kourovkanotebookorg.wordpress.com}

\bibitem{mm} G. Malle, B.\,H. Matzat, Inverse {Galois} theory, 2nd ed., Springer Monogr. Math., Berlin: Springer (2018). DOI: \href{https://doi.org/10.1007/978-3-662-55420-3}{\texttt{10.1007/978-3-662-55420-3}}

\bibitem{R11} D.\,O. Revin, On Baer--Suzuki $\pi$-theorems, {\em Sib. Math. J.}, {\bf 52}, N\,2 (2011), 340--347. DOI: \href{https://doi.org/10.1134/S0037446611020170}{\texttt{10.1134/S0037446611020170}}

\bibitem{RZ1} D.\,O. Revin, A.\,V. Zavarnitsine, On generations by conjugate elements in almost simple groups with socle~${}^2F_4(q^2)'$,
\textit{J. Group Theory},   {\bf 27}, N\,1 (2024), 119--140. DOI: \href{https://doi.org/10.1515/jgth-2022-0216}{\texttt{10.1515/jgth-2022-0216}}
 
\bibitem{RZ25} D.\,O. Revin, A.\,V. Zavarnitsine,  Conjugate generation of sporadic almost simple groups, {\em Sib. Elect. Math. Reports}, {\bf 22}, N\,1 (2025), 552--562. DOI: \href{https://doi.org/10.33048/semi.2025.22.037}{\texttt{10.33048/semi.2025.22.037}}

\bibitem{RZG} D.\,O. Revin, A.\,V. Zavarnitsine, GAP code accompanying this paper  (2026). URL: \url{https://github.com/zavandr/refined-sporadic}

\bibitem{Scott} L.\,L. Scott, Matrices and cohomology, {\em Ann. Math.}, {\bf 105} (1977), 473--492. DOI: \href{https://doi.org/10.2307/1970920}{\texttt{10.2307/1970920}}

\bibitem{WGR} Zh. Wang, W. Guo, D.\,O. Revin, Toward a sharp Baer--Suzuki theorem for the $\pi$-radical: exceptional groups of small rank, {\em  Algebra and Logic}, {\bf 62}, N1 (2023), 1--21. DOI: \href{https://doi.org/10.1007/s10469-023-09720-3}{\texttt{10.1007/s10469-023-09720-3}}

\bibitem{Wil96} R.\,A. Wilson, Standard generators for sporadic simple groups, {\em J. Algebra}, {\bf 184}, N\,2 (1996),  505--515.  DOI: \href{https://doi.org/10.1006/jabr.1996.0271}{\texttt{10.1006/jabr.1996.0271}}

\bibitem{AtlRep} R.\,A. Wilson,  P.\,G. Walsh, J.\,Tripp, I.\,A.\,I. Suleiman, S.\, Rogers, R.\,A. Parker, S.\,P. Norton, S.\,J. Nickerson, S.\,A. Linton, J.\,N. Bray and R.\,A. Abbott, Atlas of Finite Group Representation (2025), Accessed: 2026-02-01. URL: \url{http://brauer.maths.qmul.ac.uk/Atlas/}

\bibitem{YRV} N. Yang, D.\,O. Revin, E.\,P. Vdovin, Baer--Suzuki theorem for the $\pi$-radical,
\textit{Israel J. Math.}, {\bf 245}, N\,1 (2021), 173--207. DOI: \href{https://doi.org/10.1007/s11856-021-2209-y}{\texttt{10.1007/s11856-021-2209-y}}

\bibitem{YWR} N. Yang, Zh. Wu, D.\,O.  Revin,  On the sharp Baer–Suzuki theorem for the 
$\pi$-radical: Sporadic groups, {\em Sib. Math. J.}, {\bf 63}, N\,2 (2022), 387--394.
DOI: \href{https://doi.org/10.1134/S0037446622020161}{\texttt{10.1134/S0037446622020161}}

\bibitem{YWRV} N. Yang, Zh. Wu, D.~O. Revin, E.~P. Vdovin, On the sharp Baer--Suzuki theorem for the $\pi$-radical of a~finite group,
\textit{Sb. Math.}, {\bf 214}, N\,1 (2023), 108--147. DOI:
\href{https://doi.org/10.4213/sm9698e}{\texttt{10.4213/sm9698e}}

\end{thebibliography}
\end{document}